\documentclass{amsart}

\usepackage[english]{babel}
\usepackage[colorlinks]{hyperref}
\usepackage[dvipsnames]{xcolor}
\usepackage[shortlabels]{enumitem}
\setenumerate{label=(\roman*)}

\newcommand\myshade{100}
\hypersetup{
  linkcolor  = RoyalBlue!\myshade!black,
  citecolor  = ForestGreen!\myshade!black,
  urlcolor   = RedOrange!\myshade!black,
  colorlinks = true,
}

\usepackage{amsmath}
\usepackage{amssymb}
\usepackage{amsfonts}
\usepackage{amsthm}
\usepackage{mathtools}
\usepackage{bm}
\usepackage{bbold}

\usepackage[capitalise]{cleveref}
\crefformat{equation}{(#2#1#3)}
\crefname{subsection}{Section}{Sections}
\crefname{prop}{Proposition}{Propositions}

\theoremstyle{plain}
\newtheorem{thm}{Theorem}[section]
\newtheorem{lem}[thm]{Lemma}
\newtheorem{prop}[thm]{Proposition}
\newtheorem{cor}[thm]{Corollary}

\theoremstyle{definition}
\newtheorem{defn}[thm]{Definition}
\newtheorem{example}[thm]{Example}

\theoremstyle{remark}
\newtheorem*{rem}{Remark}

\makeatletter
\def\@crefthmfix#1{\AddToHook{env/#1/begin}{\crefalias{thm}{#1}}}
\@for\@tmpenv:={thm,lem,prop,cor,defn,example}\do{%
  \expandafter\@crefthmfix\expandafter{\@tmpenv}}
\makeatother

\makeatletter
\def\@fnsymbol#1{\ensuremath{\ifcase#1\or \dagger\or \ddagger\or
           \dagger\dagger
           \or \ddagger\ddagger \else\@ctrerr\fi}}
\makeatother

\newcommand\N{\ensuremath{\mathbb{N}}}
\newcommand\R{\ensuremath{\mathbb{R}}}

\usepackage[backend=biber,maxnames=999,giveninits=true,style=trad-plain,sortcites=true,datamodel=mrnumber,isbn=false,url=false,doi=false]{biblatex}
\usepackage{csquotes}

\DeclareFieldFormat{title}{\myhref{\mkbibemph{#1}}}
\DeclareFieldFormat
  [article,inproceedings,unpublished]
  {title}{\myhref{\mkbibquote{#1\isdot}}}

\newcommand{\myhref}[1]{%
  \ifhyperref
    {\iffieldundef{doi}
       {\iffieldundef{url}
          {#1}
          {\href{\strfield{url}}{#1}}}
       {\href{http://dx.doi.org/\strfield{doi}}{#1}}}
    {#1}%
}

\DeclareFieldFormat{mrnumber}{%
  MR\addcolon\space
  \ifhyperref
    {\href{http://www.ams.org/mathscinet-getitem?mr=#1}{\nolinkurl{#1}}}
    {\nolinkurl{#1}}}

\renewbibmacro*{doi+eprint+url}{%
  \printfield{mrnumber}%
  \ifentrytype{online}
    {\newunit\newblock\usebibmacro{url+urldate}}
    {}%
}

\newcommand{\one}{\mathbb{1}}
\newcommand{\zero}{\mathbb{0}}
\newcommand{\falg}{$f\!$-algebra}
\newcommand{\fsubalg}{$f\!$-subalgebra}
\newcommand{\FVL}[1]{\operatorname{FVL}(#1)}
\newcommand{\FAFA}[1]{\operatorname{FA{\it f}\!A}(#1)}
\newcommand{\uFAFA}[1]{\operatorname{FA{\it f}\!A}^1(#1)}
\newcommand{\FUCVL}[1]{\operatorname{FUCVL}(#1)}
\newcommand{\FUCFA}[1]{\operatorname{FUC{\it f}\!A}(#1)}
\newcommand{\uFUCFA}[1]{\operatorname{FUC{\it f}\!A}^1(#1)}
\newcommand{\ulim}[3]{#1 \xrightarrow{\rm u} #2\;(#3)}
\newcommand{\VLA}{\operatorname{VLA}}
\newcommand{\Phx}{\Phi_{\bm x}}
\newcommand{\Psx}{\Psi_{\bm x}}

\title[Polynomial growth continuous function calculus]{Lattice-ordered
algebras admitting a polynomial growth continuous function calculus}
\author{David Muñoz-Lahoz}
\address{Instituto de Ciencias Matemáticas\\Universidad Autónoma de
Madrid}
\email{david.munnozl (at) uam (dot) es}
\thanks{Research supported by an FPI–UAM 2023 contract (funded by
Universidad Autónoma de Madrid) and by grants PID2024-162214NB-I00 and
CEX2023-001347-S (funded by
MCIN/\allowbreak AEI/\allowbreak 10.13039/\allowbreak 501100011033).}
\date{\today}

\begin{document}

\begin{abstract}
We characterize the Archimedean lattice-ordered algebras with identity
that admit a polynomial growth continuous function calculus. More
precisely, for an $n$-tuple $\bm x=(x_1,\dots,x_n)$ in an Archimedean
lattice-ordered algebra $X$ with identity $1_X$, we prove that the
existence of a lattice-algebra homomorphism from the algebra $PG_n$ of
continuous functions on $\mathbb{R}^n$ of polynomial growth, sending
the coordinate projections to $x_1,\dots,x_n$ and the constant
function $\one$ to $1_X$, is equivalent to the existence of $f\ge
1_X\vee |x_1|\vee \cdots \vee |x_n|$ and an \fsubalg\ $Y$ of $X$ such
that $1_X,x_1,\ldots ,x_n \in Y$ and, for every $m \in \N$, the norm
$\|{\cdot }\|_{f^{m}}$ is complete on $Y\cap I_{f^{m}}$. This result
may be viewed as an analogue, for lattice-ordered algebras, of the
characterization of positively homogeneous continuous function
calculus for Archimedean vector lattices due to Laustsen and Troitsky.
As a by-product, we describe the finitely generated free objects in
the category of uniformly complete Archimedean \falg s and also show that
the existence of a nontrivial polynomial growth continuous function
calculus on a vector space forces it to be a commutative \falg .
\end{abstract}

\maketitle

\section{Introduction}

In \cite{laustsen_troitsky2020}, N.\ J.\ Laustsen and V.\ G.\ Troitsky
characterized the Archimedean vector lattices that admit a positively
homogeneous function calculus. More precisely, they showed that if $X$
is an Archimedean vector lattice, and $\bm x=(x_1,\ldots ,x_n) \in
X^{n}$ for some $n \in \N$, then the following are equivalent:
\begin{itemize}
    \item There is a vector lattice homomorphism $\Phi _{\bm x}\colon
        H_n\to X$ with $\Phi _{\bm x}(\pi _i)=x_i$ for $i=1,\ldots
        ,n$, where $H_n$ denotes the vector lattice of positively
        homogeneous continuous functions defined on $\R^{n}$, and $\pi
        _i\colon \R^{n}\to \R$ is the $i$-th coordinate projection.
    \item There exists $f \in X$ such that $f\ge |x_1|\vee \cdots \vee
        |x_n|$ and the norm
        \[
        \|x\|_f=\inf \{\, \lambda \in [0,\infty ) : |x|\le \lambda f
        \, \},
        \]
        defined on the order ideal $I_f$ generated by $f$, is complete
        when restricted to the closed sublattice generated by
        $x_1,\ldots ,x_n$ in $I_f$.
\end{itemize}

The goal of this paper is to characterize, in a similar vein, the
Archimedean lattice-ordered algebras with identity that admit a
polynomial growth continuous function calculus. More precisely, let $X$
be an Archimedean lattice-ordered algebra with identity $1_X$, let $\bm
x=(x_1,\ldots ,x_n) \in X^{n}$ for some $n \in \N$, and denote by
$PG_n$ the lattice-ordered algebra formed by the continuous
functions $f\colon \R^{n}\to \R$ for which there exist $N \in \N$ and
$M \in \R_+$ such that $|f|\le M(\one + |\pi _1| + \cdots +|\pi
_n|)^{N}$ (where $\one\colon \R^{n}\to \R$ denotes the function that
is identically one). Then the following are equivalent:
\begin{itemize}
    \item There is a lattice-algebra homomorphism $\Psi _{\bm x}\colon
        PG_n\to X$ with $\Psi _{\bm x}(\one)=1_X$ and $\Psi _{\bm
        x}(\pi _i)=x_i$ for $i=1,\ldots ,n$.
    \item There exist $f\ge 1_X \vee |x_1| \vee \cdots \vee |x_n|$ and
        an \fsubalg\ $Y$ of $X$ such that $1_X,x_1,\ldots ,x_n \in Y$
        and, for every $m \in \N$, the norm $\|{\cdot }\|_{f^{m}}$ is
        complete on $Y\cap I_{f^{m}}$.
\end{itemize}
Compared to the Archimedean vector lattice setting, the condition at
hand is notably more intricate. The extra algebraic layer (and the
passage to algebra homomorphisms) is not really what drives this added
complexity; the real source is that $PG_n$ has a richer internal
structure than $H_n$. Recall that $H_n$ is an AM-space with unit, a
feature that keeps the corresponding conditions tidy. While $PG_n$
itself fails to be an AM-space with unit, it does decompose as a
countable union of such spaces. Concretely, setting $d = \one \vee
|\pi_1| \vee \cdots \vee |\pi_n|$, one has
  \begin{equation*}
      PG_n = \bigcup_{m \in \N} I_{d^{m}}, \qquad \text{with} \qquad I_{d^{m}} \subseteq I_{d^{m+1}}.
  \end{equation*}
This particular structure will be central in proving that the two
conditions above are equivalent. The appearance of an \fsubalg\ in the
condition is also expected, since \falg s are the only lattice-ordered
algebras admitting a polynomial growth continuous function calculus
(\cref{prop:admitPG}).

As noted in \cite{laustsen_troitsky2020},
positively homogeneous continuous function calculus plays a
fundamental role in the theory of Banach lattices. This form of
function calculus was first introduced by Yudin \cite{yudin1939} and
Krivine \cite{krivine1974} for
Banach lattices, and was subsequently generalized to uniformly
complete Archimedean vector lattices in a seminal paper by Buskes, de
Pagter and van Rooij \cite{buskes_depagter_vanrooij1991}. In that work, the authors deal not only with
uniformly complete Archimedean vector lattices, but also with
uniformly complete Archimedean \falg s. It is in the latter context
that the existence of a polynomial growth continuous function calculus
is established. Hence, just as Laustsen and Troitsky did for Archimedean
vector lattices, it seems relevant to provide a useful characterization
of the Archimedean lattice-ordered algebras that admit such a function
calculus.

\Cref{sec:why} provides two further arguments to justify why
polynomial growth continuous function calculus is the natural function
calculus to consider in lattice-ordered algebras. The first shows that
this function calculus is the natural extension of positively
homogeneous continuous function calculus on vector lattices. The
second relies on viewing function calculus through the lens of free
objects. Although the latter is an approach certainly known and used
by experts, we have not found it in the literature. As part of this
approach, in \cref{thm:fucfa} we characterize the free object
generated by finitely many elements in the category of uniformly
complete Archimedean \falg s.

\Cref{sec:main} contains the main result of this paper,
\cref{thm:conditions}, which provides the equivalent conditions
described at the beginning of this introduction.
As a consequence, we obtain in \cref{cor:conditions} a
characterization of the Archimedean lattice-ordered algebras with
identity admitting a polynomial growth continuous function calculus.
We also use examples to illustrate the different ways in which a
lattice-ordered algebra may fail to admit such a function
calculus.

Finally, in \cref{sec:final}, we prove that if a vector space admits a
sufficiently strong polynomial growth continuous function
calculus (without constants), then it is automatically a commutative \falg. This result is
an analogue of \cite[Theorem 1.4]{laustsen_troitsky2020}.

\subsection{Preliminaries and notation}

For the unexplained terminology regarding vector and Banach lattices,
we refer the reader to the classical monographs
\cite{aliprantis_burkinshaw2006,meyer-nieberg1991}. An accessible
introduction to lattice-ordered algebras and \falg s can be found in
\cite{huijsmans1991}. We will refer to the subspaces of a lattice-ordered algebra that
are at the same time vector sublattices and subalgebras as
\emph{sublattice-algebras}. When a sublattice-algebra is itself an
\falg, we will call it an \emph{\fsubalg}.

Let $X$ be an Archimedean vector lattice. That a net
$(x_\alpha )\subseteq X$ converges uniformly to $x$ with regulator $e
\in X_+$ will be denoted by $\ulim{x_\alpha }{x}{e}$. We denote by
$I_e$ the order ideal generated by $e$, and by $\|{\cdot }\|_e$ the
associated norm. An Archimedean vector lattice is \emph{uniformly
complete} if $(I_e,\|{\cdot }\|_e)$ is norm complete for every $e \in
X_+$; although uniform completeness can also be defined for
non-Archimedean vector lattices, all uniformly complete vector
lattices in this paper will be Archimedean.

\section{Why polynomial growth continuous function
calculus?}\label{sec:why}

The goal of this section is to introduce polynomial growth continuous
function calculus in lattice-ordered algebras with identity and to
justify, in two ways, why it is the natural function calculus to
consider in this category. We also show that every lattice-ordered
algebra having such a function calculus must be an \falg.

\subsection{Polynomial growth continuous function calculus as a
natural extension of positively homogeneous continuous function
calculus}

A real-valued function $f$ defined on a vector space $X$ is said to be
\emph{positively homogeneous} if, for every $x \in X$ and every $\lambda \in
\R_+$, the identity $f(\lambda x)=\lambda f(x)$ holds. For $n \in \N$, denote by $H_n$ the
vector lattice of positively homogeneous continuous functions defined
on $\R^{n}$. It is standard that $H_n$ is lattice isomorphic to the closed
sublattice of $C([-1,1]^{n})$ formed by the positively homogeneous
functions, the isomorphism being the map that sends a function to its
restriction to $[-1,1]^{n}$. Pulling back the norm from
$C([-1,1]^{n})$, $H_n$ becomes a Banach lattice. In particular, $H_n$ is uniformly complete.

For $i \in \{1,\ldots ,n\}$, the $i$-th coordinate projection
\[
\begin{array}{cccc}
\pi _i\colon& \R^{n} & \longrightarrow & \R \\
        & (t_1,\ldots ,t_n) & \longmapsto & t_i \\
\end{array}
\]
belongs to $H_n$. We shall denote the $i$-th coordinate
projection by $\pi _i$ regardless of which $\R^{n}$ is its domain. This
abuse of notation is justified if we think of $\R^{n}$ as the first
$n$ coordinates of $\R^{\N}$, and of $\pi _i$ as the restriction to
$\R^{n}$ of the $i$-th coordinate projection in $\R^{\N}$.

Buskes, de Pagter and van Rooij showed in \cite[Theorem 3.7
]{buskes_depagter_vanrooij1991} that every uniformly complete
Archimedean vector lattice admits a positively homogeneous continuous
function calculus. More precisely, they showed that if $X$ is a
uniformly complete Archimedean vector lattice, and $\bm
x=(x_1,\ldots ,x_n)\in X^{n}$ for some $n \in \N$, then there exists a unique lattice
homomorphism $\Phx\colon H_n\to X$ satisfying $\Phx(\pi _i)=x_i$ for
$i=1,\ldots ,n$.

What is the natural extension of this function calculus to
lattice-ordered algebras? Suppose now that $X$ is an
Archimedean uniformly complete lattice-ordered algebra.
Let $\bm x=(x_1,\ldots ,x_n) \in X^{n}$ and let $\Phx\colon H_n\to X$
be the unique lattice homomorphism introduced above. If we want to do
function calculus in lattice-ordered algebras, the least we should ask
is for the map to be a lattice-algebra homomorphism. But $H_n$ is not an
algebra. Still, we can view $H_n$ as a sublattice of $\R^{\R^{n}}$, the
space of all functions $\R^{n}\to \R$ equipped with pointwise linear
and lattice structures. This space, when also equipped with the pointwise
product, is a lattice-ordered algebra, with identity the constant one
function $\one$. The weakest extension of the positively homogeneous continuous
function calculus to the algebraic setting would be a
lattice-algebra homomorphism $\Psx\colon \VLA(H_n)\to X$, where
$\VLA(H_n)$ is the smallest sublattice-algebra of $\R^{\R^{n}}$
containing $H_n$.
Since $X$ is uniformly complete and, as the proof of
\cref{prop:admitPG} shows, necessarily an \falg, we automatically
obtain something stronger:
$\Psx$ extends uniquely to a lattice-algebra homomorphism
on the smallest uniformly complete sublattice-algebra of $\R^{\R^{n}}$
that contains $H_n$ (this is proved in \cref{lem:ucfalg}). Thus the
question becomes: can we identify the smallest uniformly complete
sublattice-algebra of $\R^{\R^{n}}$ that contains $H_n$?

A function $f\colon \R^{n}\to
\R$ is said to have \emph{polynomial growth} if there exist $N \in \N$ and $M \in \R_+$ such that
\[
    |f|\le M(\one+|\pi _1|+ \cdots +|\pi _n|)^{N}.
\]
The lattice-ordered algebra formed by the continuous real-valued
functions on $\R^{n}$ that have polynomial growth is denoted by $PG_n$. Consider the elements $f \in PG_n$ for
which the limit
\[
\lim_{t\to 0^{+}}\frac{f(tx)}{t}
\]
exists for every $x \in \R^{n}$, with uniform convergence on
bounded subsets.
In this case, we denote $f_h(x)=\lim_{t\to
0^{+}}f(tx)t^{-1}$ and say that $f_h$ exists. Observe that $f_h \in H_n$,
hence the notation. Denote by $PG_n^{h}$ the set of elements of $PG_n$
for which $f_h$ exists.

\begin{thm}[{\cite[Propositions 3.6 and
    4.9]{buskes_depagter_vanrooij1991}}]\label{thm:ucgenerated}
    Let $n \in \N$.
    \begin{enumerate}
        \item The uniformly complete vector lattice generated by $\pi
            _1,\ldots ,\pi _n$ in $\R^{\R^{n}}$ is $H_n$.
        \item The uniformly complete lattice-ordered algebra generated
            by $\pi _1,\ldots ,\pi _n$ in $\R^{\R^{n}}$ is $PG_n^{h}$.
        \item The uniformly complete lattice-ordered algebra generated
            by $\one,\pi _1,\ldots ,\pi _n$ in $\R^{\R^{n}}$ is $PG_n$.
    \end{enumerate}
\end{thm}

We were looking for the smallest uniformly complete lattice-ordered
algebra in $\R^{\R^{n}}$ containing $H_n$. According to the previous
theorem, this is precisely $PG_n^{h}$.
And the smallest uniformly complete lattice-ordered algebra in
$\R^{\R^{n}}$ containing $H_n$ and $\one$ is $PG_n$. In other words,
the weakest natural extension of the positively homogeneous
continuous function calculus to the algebraic setting is a
lattice-algebra homomorphism $\Psx^{h}\colon PG_n^{h}\to X$ (or a
lattice-algebra homomorphism $\Psx\colon PG_n\to X$ satisfying
$\Psx(\one)=1_X$, in the case that
$X$ has an identity $1_X$).

Inspired by \cite{laustsen_troitsky2020},
the main goal of this paper is to study which unital
lattice-ordered algebras admit a polynomial growth continuous
function calculus, meaning the following:

\begin{defn}
    A lattice-ordered algebra $X$ with identity $1_X$ admits a \emph{polynomial
    growth continuous function calculus} if, for each $n \in \N$ and
    each $\bm x=(x_1,\ldots ,x_n)\in X^{n}$, there is a
    lattice-algebra homomorphism $\Psx\colon PG_n\to X$ satisfying
    \begin{equation}\label{eq:proj_one}
            \Psx(\one)=1_X \text{ and }
\Psx(\pi _i)=x_i\text{ for }i=1,\ldots ,n.
    \end{equation}
\end{defn}

This is the analogue of \cite[Definition 1.1]{laustsen_troitsky2020},
where the authors defined a vector lattice
that admits a positively homogeneous continuous function calculus
as a vector lattice $X$ for which, for every $n \in \N$ and $\bm
x=(x_1,\ldots ,x_n)\in X^{n}$, there is a vector lattice homomorphism
$\Phx\colon H_n\to X$ satisfying $\Phx(\pi _i)=x_i$ for $i=1,\ldots ,n$.

Our focus will be on the case with identity, where we are truly dealing with
all functions of polynomial growth. However, the following is a
natural definition that will be useful at times.

\begin{defn}
    A lattice-ordered algebra $X$ admits a \emph{polynomial
    growth continuous function calculus without constants} if, for each $n \in \N$ and
    each $\bm x=(x_1,\ldots ,x_n)\in X^{n}$, there is a
    lattice-algebra homomorphism $\Psx^{h}\colon PG_n^{h}\to X$ satisfying
    \begin{equation}\label{eq:proj}
    \Psx^{h}(\pi _i)=x_i\quad\text{for }i=1,\ldots ,n.
    \end{equation}
\end{defn}

The following is a routine, but relevant, observation.

\begin{prop}\label{prop:admitPG}
    Let $X$ be a lattice-ordered algebra that admits a polynomial
    growth continuous function calculus without constants. Then $X$ is a commutative
    \falg.
\end{prop}
\begin{proof}
    Let $X$ be a lattice-ordered algebra that admits a polynomial
    growth continuous function calculus without constants. Let $x_1,x_2,x_3\in X_+$,
    and suppose $x_1\wedge x_2=0$. Let $\Psi \colon PG_3^{h}\to X$ be a
    lattice-algebra homomorphism such that $\Psi (\pi _i)=x_i$ for
    $i=1,\ldots ,3$. Since $\Psi $ is a lattice-algebra homomorphism:
    \begin{align*}
        \Psi (|\pi _1|-|\pi _1|\wedge |\pi _2|)&=|x_1|-|x_1|\wedge
        |x_2|=x_1,\\
        \Psi (|\pi _2|-|\pi _1|\wedge |\pi _2|)&=|x_2|-|x_1|\wedge
        |x_2|=x_2,\\
        \Psi (|\pi _3|)&=|x_3|=x_3.
    \end{align*}
    Since
    \[
        (|\pi _1|-|\pi _1|\wedge |\pi _2|)\wedge (|\pi
        _2|-|\pi _1|\wedge |\pi _2|)=0
    \]
    and $PG_3^{h}$ is an \falg, it follows that
    \[
        [|\pi _3|(|\pi _1|-|\pi _1|\wedge |\pi _2|)]\wedge (|\pi
        _2|-|\pi _1|\wedge |\pi _2|)=0.
    \]
    Applying $\Psi$ to the previous equality yields $(x_3x_1)\wedge
    x_2=0$. Using the same idea, it is straightforward to check that $X$
    is commutative. Hence $X$ is a commutative \falg.
\end{proof}

\begin{rem}
    Since $PG_n^{h}$ is a sublattice-algebra of $PG_n$, a
    lattice-ordered algebra admitting a polynomial growth continuous
    function calculus also admits one without constants; hence
    \cref{prop:admitPG} applies to it as well.
\end{rem}

The previous proposition can be seen as yet another reason why
\falg s are particularly well-behaved and relevant among all
lattice-ordered algebras.

To close this circle of ideas, let us note that, just like every
Archimedean uniformly complete vector lattice admits a positively
homogeneous continuous function calculus, every Archimedean uniformly
complete \falg\ with identity admits a polynomial growth continuous
function calculus \cite[Theorem
4.12]{buskes_depagter_vanrooij1991}. We shall reprove this in the
next section using an approach completely different to the one used in the
referenced paper.

\subsection{Function calculus through free objects}

In this section we show how to approach function calculus through free
objects. This way of understanding function calculus will be useful in
our later proofs, and provides a unified approach to positively
homogeneous and polynomial growth (without constants) continuous function
calculi. Moreover, it further justifies why polynomial growth
continuous function calculus is the natural function calculus to
consider on \falg s.

Recall that a \emph{free vector lattice} over a non-empty set $S$ is a
pair $(\FVL S,\delta)$, where $\FVL S$ is a vector lattice and $\delta
\colon S\to \FVL S$ is a map, such that, for every vector lattice $X$
and every map $T\colon S\to X$, there exists a unique lattice
homomorphism $\hat{T}\colon \FVL S\to X$ satisfying $\hat{T}\delta =T$.
The existence of this object follows from general results of universal
algebra \cite[pp.~143--144]{birkhoff1967}; it was later
constructed and studied in \cite{baker1968,bleier1973}.

The free vector lattice over a set $S$ is essentially unique and can
be constructed explicitly in the following way: for every $s \in S$, define the point evaluation $\delta _s\colon \R^{S}\to \R$ by
$\delta _s(x)=x(s)$; then the free vector lattice is the sublattice generated by
$\{\delta _s:s \in S\}$ in $\R^{\R^{S}}$ together with the map
$\delta\colon S\to \FVL S$ defined by $\delta (s)=\delta _s$.

Note that, with this particular realization, every element of $\FVL S$ is
continuous in the product topology; hence $\FVL S$ lives inside
$C(\R^{S})$. Moreover, every element of $\FVL S$ is positively
homogeneous; hence $\FVL S$ lives inside $H(\R^{S})$, the vector
lattice of positively homogeneous continuous functions defined on $\R^{S}$.
Finally, restricting the functions from $\R^{S}$ to $\Delta
_S=[-1,1]^{S}$, $\FVL S$ can be regarded as a vector sublattice of
$H(\Delta _S)$ (see \cite[Lemma 5.1]{de_pagter_wickstead2015} for the
details). Note that the latter space, being a closed sublattice
of $C(\Delta _S)$, is a Banach lattice.

One can similarly define the free object in the category of uniformly
complete vector lattices. This notion was first considered in
\cite{emelyanov_gorokhova2024}.

\begin{defn}
    Let $S$ be a non-empty set. A \emph{free uniformly complete vector
    lattice} over $S$ is a pair $(\FUCVL S, \delta)$, where $\FUCVL S$
    is a uniformly complete vector lattice and $\delta \colon S\to
    \FUCVL S$ is a map, such that, for every uniformly complete vector
    lattice $X$ and every map $T\colon S\to X$, there exists a unique
    lattice homomorphism $\hat{T}\colon \FUCVL S\to X$ satisfying
    $T=\hat{T}\circ \delta $.
\end{defn}

In \cite{emelyanov_gorokhova2024} an abstract construction of this
object is provided through the uniform completion (called
r-completion by the authors): a
uniformly complete vector lattice $U$ is a \emph{uniform completion} of a
vector lattice $X$ if there exists a vector lattice embedding $i\colon X\to
U$ such that, for each uniformly complete vector lattice $Y$ and each
lattice homomorphism $T\colon X\to Y$, there exists a unique lattice
homomorphism $S\colon U\to Y$ satisfying $T=S\circ i$. Uniform completions
are unique up to a unique lattice isomorphism. It is known (see
\cite{emelyanov_gorokhova2024} for the details) that, if $X$ is a
vector sublattice of $Y$, and $Y$ is uniformly complete, then the
uniform completion of $X$ is the intersection of all the uniformly
complete vector sublattices of $Y$ containing $X$.

By checking the universal properties, it is clear that $\FUCVL S$
is the uniform completion of $\FVL S$. Since $H(\Delta _S)$ is a
Banach lattice, it is a uniformly complete vector lattice, and
therefore $\FUCVL S$ is the intersection of all the uniformly complete
vector sublattices of $H(\Delta _S)$ that contain $\FVL S$. It follows
from \cref{thm:ucgenerated} that, if $S$ is finite, then $\FUCVL
S=H(\Delta _S)$, which is in turn isomorphic to $H_{|S|}$. Under this
isomorphism, and the identification of $S$ with $\{1,\ldots ,|S|\}$, the $\delta _s$ become the projections $\pi _i$.
(Note that in this case the free uniformly complete vector lattice
coincides with the free Banach lattice over $S$, see
\cite{de_pagter_wickstead2015,aviles_rodriguez_tradacete2018}).

Now we can establish the relation of this theory with positively
homogeneous continuous function calculus. As usual, let $X$ be a uniformly
complete Archimedean vector lattice, let $n \in \N$ and let $\bm
x=(x_1,\ldots ,x_n) \in X^{n}$. The universal property of
$\FUCVL{\{1,\ldots ,n\}}$ asserts that there exists a unique lattice
homomorphism
\[
\FUCVL{\{1,\ldots ,n\}}\to X
\]
sending $\delta _i$ to $x_i$ for $i=1,\ldots ,n$. But, as noted above,
$\FUCVL{\{1,\ldots ,n\}}$ is essentially $H_n$ and, with this
identification, $\delta _i$ becomes $\pi _i$. Hence the above map is
the unique lattice homomorphism $\Phx\colon H_n\to X$ satisfying
$\Phx(\pi _i)=x_i$ for $i=1,\ldots ,n$.

So far we have only revisited well-known facts. However, looking back at this
approach, there is nothing preventing us from applying it to
\falg s. What kind of function calculus will appear in this case? Of
course, it will be the polynomial growth continuous function calculus. The
details follow.

A \emph{free Archimedean \falg} over a non-empty set $S$ is a pair $(\FAFA S,
\delta )$, where $\FAFA S$ is an Archimedean \falg\ and $\delta \colon S\to \FAFA
S$ is a map, such that, for every Archimedean \falg\ $A$ and every map $T\colon
S\to A$, there exists a unique lattice-algebra homomorphism
$\hat{T}\colon \FAFA S\to A$ satisfying $\hat{T}\delta =T$.
Similarly, a \emph{free Archimedean \falg\ with identity} over a non-empty set
$S$ is a pair $(\uFAFA S,
\delta )$, where $\uFAFA S$ is an Archimedean \falg\ with identity and $\delta
\colon S\to \uFAFA
S$ is a map, such that, for every Archimedean \falg\ with identity $A$ and every map $T\colon
S\to A$, there exists a unique lattice-algebra homomorphism
$\hat{T}\colon \uFAFA S\to A$ that sends the identity to the identity and
satisfies $\hat{T}\delta =T$. The
existence of these objects follows from the works
\cite{birkhoff1942,henriksen_isbell1962}.

The free Archimedean \falg\ can be constructed explicitly as
the sublattice-algebra generated by $\{\, \delta _s : s \in S \, \} $
in $\R^{\R^{S}}$. Similarly, the free Archimedean \falg\ with identity can
be identified with the sublattice-algebra generated by $\{\, \delta _s
: s \in S \, \} \cup \{\one\}$. These facts follow from the works of
Birkhoff \cite{birkhoff1935} and Henriksen and Isbell
\cite{henriksen_isbell1962} in the realm of universal algebras; see
\cite[Chapter 2]{munoz-lahoz_tradacete2025} for a modern treatment
of this theory that may be more amenable to the working functional
analyst. Note that, with this particular
realization, every element of $\uFAFA S$ is continuous in the product
topology; hence $\FAFA S$ and $\uFAFA S$ live inside $C(\R^{S})$. However, the
parallelism with the free vector lattice ends here. The functions of $\FAFA S$
cannot be restricted to $\Delta _S$ injectively since they are not,
in general, positively homogeneous: there exist non-zero elements of
$\FAFA S$ that vanish on $\Delta _S$.

One can still consider the free object in the category of uniformly
complete \falg s.

\begin{defn}
    Let $S$ be a non-empty set. A \emph{free uniformly complete \falg}
    (resp., \emph{free uniformly complete \falg\ with identity})
    over $S$ is a pair $(\FUCFA S, \delta )$ (resp., $(\uFUCFA S,\delta
    )$), where $\FUCFA S$ (resp., $\uFUCFA S$) is a
    uniformly complete \falg\ (resp., uniformly complete \falg\ with identity) and
    $\delta \colon S\to \FUCFA S$ (resp., $\delta \colon S\to \uFUCFA S$) is a
    map, such that, for every uniformly complete \falg\ (resp., \falg\ with
    identity) $A$ and every
    map $T\colon S\to A$, there exists a unique lattice-algebra
    homomorphism $\hat{T}\colon \FUCFA S\to A$ (resp., $\hat{T}\colon
    \uFUCFA S\to A$) satisfying $T=\hat{T}\circ \delta $ (resp., and
    sending the identity to the identity).
\end{defn}

A priori, it is not clear at all whether such an object exists. The
following lemma allows for an abstract construction of the free
uniformly complete \falg\ as the uniform completion of $\FAFA S$.

\begin{lem}\label{lem:ucfalg}
    Let $Y$ be a uniformly complete \falg\ and let $X$ be a
    sublattice-algebra of $Y$. Let $\bar{X}$ be the intersection of
    all the uniformly complete vector sublattices of $Y$ containing
    $X$ (i.e., $\bar{X}$ is the uniform completion of $X$). Then
    \begin{enumerate}
        \item $\bar{X}$ is a sublattice-algebra (and therefore an
            \falg),
        \item for every uniformly complete \falg\ $A$ and every
            lattice-algebra homomorphism $T\colon X\to A$ there exists a
            unique extension of $T$ to a lattice-algebra homomorphism
            $\bar{T}\colon \bar{X}\to A$.
    \end{enumerate}
\end{lem}
\begin{proof}
    Item \textit{(i)} is \cite[Lemma 2.1]{buskes_depagter_vanrooij1991}.
    To show \textit{(ii)}, let $A$ be a uniformly complete \falg\ and
    let $T\colon X\to A$ be a lattice-algebra homomorphism. Since $\bar{X}$ is
    the uniform completion of $X$, $T$ extends to a unique lattice
    homomorphism $\bar{T}\colon \bar{X}\to A$. The goal is to show
    that $\bar{T}$ is an algebra homomorphism. Let $x \in Y_+$. Since
    $Y$ is an \falg, the left multiplication operator $L_x\colon Y\to
    Y$, $L_x(y)=xy$, is a lattice homomorphism. If $x \in X_+$, then
    we can consider its restriction $L_x|_{X}\colon X\to X$. Since
    $T$ is an algebra homomorphism, $T\circ L_x|_X=L_{Tx}\circ T$. All
    the operators in this identity are lattice homomorphisms. Note
    that $L_x|_{\bar{X}}\colon \bar{X}\to \bar{X}$ by item
    \textit{(i)}, and that $\bar{T}\circ L_x|_{\bar{X}}$ is a lattice
    homomorphism that coincides with $T\circ L_x|_{X}$ on $X$. By
    uniqueness of the extensions, $\bar{T}\circ L_x|_{\bar{X}}$ is the
    extension of $T\circ L_x|_X$ to $\bar{X}$. Similarly,
    $L_{Tx}\circ \bar{T}$ is the extension of
    $L_{Tx}\circ T$ to $\bar{X}$. From $T\circ L_x|_X=L_{Tx}\circ T$,
    it follows that $\bar{T}\circ L_x|_{\bar{X}}=L_{Tx}\circ \bar{T}$. Hence, for
    every $x \in X_+$ and $y \in \bar{X}$,
    \[
    \bar{T}(xy)=(\bar{T}\circ L_x|_{\bar{X}})(y)=(L_{Tx}\circ
    \bar{T})(y)=Tx\, \bar{T}y.
    \]
    By linearity, this identity holds for all $x \in X$. Since the
    product is commutative, we can rewrite the previous equality as
    $\bar{T}\circ L_y|_X=L_{\bar{T}y}\circ T$, where $L_y|_X\colon
    X\to \bar{X}$. If $y \in (\bar{X})_+$, then $L_y|_X$ and
    $L_{\bar{T}y}$ are lattice homomorphisms, and by uniqueness of the
    extensions $\bar{T}\circ L_y|_{\bar{X}}=L_{\bar{T}y}\circ
    \bar{T}$. By linearity, this last equation holds for all $y \in
    \bar{X}$. Hence $\bar{T}$ is an algebra homomorphism.
\end{proof}

The previous lemma implies that $\FUCFA S$ (resp.\ $\uFUCFA S$) does
exist: it is the uniform completion of $\FAFA S$ (resp.\ $\uFAFA S$).
The following result identifies these free objects when $S$ is finite.
It follows immediately from the previous observation and
\cref{thm:ucgenerated}.

\begin{thm}\label{thm:fucfa}
    Let $n \in \N$ and let $S=\{1,\ldots ,n\}$.
    \begin{enumerate}
        \item The free uniformly complete \falg\ over
            $S$ is $PG_n^{h}$ together with the map $\delta \colon S\to
            PG_n^{h}$ given by $\delta (i)=\pi _i$ for $i=1,\ldots ,n$.
        \item The free uniformly complete \falg\ with identity over
            $S$ is $PG_n$ together with the map $\delta \colon S\to
            PG_n$ given by $\delta (i)=\pi _i$ for $i=1,\ldots ,n$.
    \end{enumerate}
\end{thm}

It is now time to revisit function calculus through the lens of
free objects. Let $X$ be a uniformly complete Archimedean \falg, let
$n \in \N$ and let $\bm x = (x_1,\ldots ,x_n) \in X^{n}$. The
universal property of $\FUCFA{\{1,\ldots ,n\}}$ asserts that there
exists a unique lattice-algebra homomorphism
\[
\FUCFA{\{1,\ldots ,n\}}\to X
\]
sending $\delta _i$ to $x_i$ for $i=1,\ldots ,n$. By the previous
theorem, $\FUCFA{\{1,\ldots ,n\}}$ can be identified with
$PG_n^{h}$, and with this identification $\delta _i$ becomes $\pi _i$.
Hence the above map is the unique lattice-algebra homomorphism
$\Psx^{h}\colon PG_n^{h}\to X$ satisfying $\Psx^{h}(\pi _i)=x_i$ for
$i=1,\ldots ,n$. In the case that $X$ has identity $1_X$, this map can be
extended to a lattice-algebra homomorphism $\Psx\colon PG_n\to X$
satisfying also $\Psx(\one)=1_X$.

Putting these observations together with \cref{prop:admitPG}, we have the
following characterization of the Archimedean uniformly complete
lattice-ordered algebras that admit a polynomial growth
continuous function calculus.

\begin{prop}\label{prop:ucPG}
    Let $X$ be an Archimedean uniformly complete lattice-ordered
    algebra.
    \begin{enumerate}
        \item $X$ admits a polynomial growth continuous function
            calculus without constants if and only if $X$ is an \falg.
        \item $X$ admits a polynomial growth continuous function
            calculus if and only if $X$ is an \falg\ with identity.
    \end{enumerate}
\end{prop}

What about general Archimedean lattice-ordered algebras? Can we
characterize those that admit a polynomial growth continuous function
calculus? The
remainder of this paper is devoted to obtaining results analogous to
those in \cite{laustsen_troitsky2020} for Archimedean lattice-ordered
algebras with identity.

\section{Necessary and sufficient conditions to admit a polynomial
growth continuous function calculus}\label{sec:main}

The goal of this section is to prove the following:

\begin{thm}\label{thm:conditions}
    Let $X$ be an Archimedean lattice-ordered algebra with identity $1_X$,
    and let $\bm x = (x_1,\ldots ,x_n)\in X^{n}$, for some $n \in \N$.
    The following are equivalent:
    \begin{enumerate}
    \item There is a lattice-algebra homomorphism $\Psx\colon
        PG_n\to X$ with $\Psx(\one)=1_X$ and $\Psx(\pi _i)=x_i$
        for $i=1,\ldots ,n$.
    \item Let $e=1_X\vee |x_1|\vee \cdots \vee |x_n|$. There exists an
        $f\!$-subalgebra $Y$ of $X$ such that $1_X,x_1,\ldots ,x_n \in
        Y$ and, for every $m \in \N$, the norm $\|{\cdot }\|_{e^{m}}$
        is complete on $Y\cap I_{e^{m}}$.
    \item There exist $f\ge 1_X\vee |x_1|\vee \cdots \vee |x_n|$ and an
        $f\!$-subalgebra $Y$ of $X$ such that $1_X,x_1,\ldots ,x_n \in
        Y$ and, for every $m \in \N$, the norm $\|{\cdot }\|_{f^{m}}$
        is complete on $Y\cap I_{f^{m}}$.
    \end{enumerate}
    When one and hence all three of these conditions are satisfied,
    the lattice-algebra homomorphism $\Psx$ satisfying the conditions
    in \textit{(i)} is unique.
\end{thm}

Compare this result with its analogue for positively homogeneous
continuous function calculus \cite[Proposition
2.2]{laustsen_troitsky2020}. We can recast the latter result to make
the analogy with \cref{thm:conditions} clearer.

\begin{prop}
    Let $X$ be an Archimedean vector lattice and let $n \in \N$. For
    $\bm x=(x_1,\ldots ,x_n) \in X^{n}$, the following are equivalent:
    \begin{enumerate}
        \item There is a vector lattice homomorphism $\Phx\colon
            H_n\to X$ with $\Phx(\pi _i)=x_i$ for $i=1,\ldots ,n$.
        \item Let $e=|x_1|\vee \cdots \vee |x_n|$. There exists a
            sublattice $Y$ of $X$ such that $x_1,\ldots ,x_n \in Y$
            and the norm $\|{\cdot
            }\|_e$ is complete on $Y\cap I_e$.
        \item There exists $f\ge |x_1|\vee \cdots
            \vee |x_n|$ and a sublattice $Y$ of $X$ such that
            $x_1,\ldots ,x_n \in Y$ and the norm
            $\|{\cdot }\|_f$ is complete on $Y\cap I_f$.
    \end{enumerate}
    When one and hence all three of these conditions are satisfied,
    the lattice homomorphism $\Phx$ satisfying the conditions
    in \textit{(i)} is unique.
\end{prop}

Of course, the conditions in \cref{thm:conditions} are more involved
than those in \cite[Proposition 2.2]{laustsen_troitsky2020}. This is
not so much due to the fact that we are dealing with the additional
algebraic structure (and algebraic homomorphisms), as it is to the
fact that the structure of $PG_n$ is more complex than that of $H_n$.
It is well-known that $H_n$ is an AM-space with unit. This allows for
much cleaner conditions. Also, the appearance of an \fsubalg\ should
not surprise the reader in view of \cref{prop:admitPG}.

Even though $PG_n$ is not an AM-space with unit, it can be expressed
as a countable increasing union of AM-spaces with unit.
More precisely, if $d=\one\vee |\pi _1|\vee \cdots \vee |\pi _n|$, then
\begin{equation*}
PG_n=\bigcup_{m \in \N} I_{d^{m}}\quad\text{and}\quad I_{d^{m}}\subseteq
I_{d^{m+1}}.
\end{equation*}
This simple observation is at the core of the proof of
\cref{thm:conditions}. Before proceeding to this proof, we need four
preliminary lemmas. The first is a slight generalization of
\cite[Lemma 2.1]{laustsen_troitsky2020}.

\begin{lem}\label{lem:prelim1}
    Let $X$ and $Y$ be Archimedean vector lattices, and let $T\colon
    X\to Y$ be a lattice homomorphism. Let $e \in X_+$. Then
    $T(I_e)\subseteq I_{Te}$, and the restriction $T|_{I_e}\colon
    (I_e,\|{\cdot }\|_e)\to (I_{Te},\|{\cdot }\|_{Te})$ is a
    contractive operator. Moreover, if $(I_e,\|{\cdot }\|_e)$ is
    complete, then the restriction of $\|{\cdot }\|_{Te}$ to
    $T(X)\cap I_{Te}$ is complete.
\end{lem}
\begin{proof}
    If $x \in I_e$, then $|x|\le \|x\|_e e$ and
    \[
    |Tx|=T|x|\le \|x\|_e Te
    \]
    because $T$ is a lattice homomorphism. Hence $\|Tx\|_{Te}\le
    \|x\|_e$. This proves that $T(I_e)\subseteq I_{Te}$ and that
    the restriction $T|_{I_e}\colon (I_e,\|{\cdot }\|_e)\to
    (I_{Te},\|{\cdot }\|_{Te})$ is a contractive operator.

    Suppose that $(I_e,\|{\cdot }\|_e)$ is complete. We will show that
    every absolutely convergent series in $(T(X)\cap I_{Te}, \|{\cdot
    }\|_{Te})$ is convergent. Let $x_n \in T(X)\cap I_{Te}$ be a
    sequence for which the associated series $\sum_{n=1}^{\infty}x_n$ is
    absolutely convergent in $\|{\cdot }\|_{Te}$. For every $n \in
    \N$, let $y_n \in X$ be such that $T(y_n)=x_n$, and set $z_n=(y_n
    \wedge \|x_n\|_{Te}e)\vee (-\|x_n\|_{Te}e)$. Then $|z_n|\le
    \|x_n\|_{Te}e$, so $z_n \in I_e$ and $\|z_n\|_{e}\le
    \|x_n\|_{Te}$. Moreover, since $|x_n|\le \|x_n\|_{Te}Te$ (i.e.,
    $-\|x_n\|_{Te}Te\le x_n\le \|x_n\|_{Te}Te$), we have $T(z_n)=(x_n
    \wedge \|x_n\|_{Te}Te)\vee (-\|x_n\|_{Te}Te)=x_n$. By assumption,
    $\sum_{n=1}^{\infty}z_n$ converges to a certain $z \in I_e$. Using
    that $T|_{I_e}$ is linear and continuous, it follows that
    $\sum_{n=1}^{\infty}x_n=Tz$.
\end{proof}

\begin{lem}\label{lem:prelim2}
    Let $X$ be an Archimedean vector lattice. For $n \in \N$, let $I_n\subseteq X$
    be a sequence of ideals such that $I_n\subseteq I_{n+1}$ and
    \[
    X=\bigcup_{n \in \N} I_n.
    \]
    Then $X$ is uniformly complete if and only if $I_n$ is uniformly
    complete for every $n \in \N$.
\end{lem}
\begin{proof}
    If $X$ is uniformly complete, then every ideal in $X$ is uniformly
    complete. Conversely, suppose that $I_n$ is uniformly complete for every $n \in \N$.
    Let $(x_k)\subseteq X$ be a Cauchy sequence in the norm $\|{\cdot }\|_y$, for some
    $y \in X_+$. By keeping only a tail of the sequence and relabeling, we may assume
    that $|x_k-x_l|\le y$ for all $k,l \in \N$. Since the sequence of
    ideals $I_n$ is increasing, there exists some $m \in \N$ for which
    $y,x_1\in I_m$. Then
    \[
    |x_k|-|x_1|\le | |x_k| - |x_1| |\le |x_k-x_1|\le y
    \]
    implies that $x_k \in I_m$ for all $k \in \N$. Since $I_m$ is
    uniformly complete, the sequence $(x_k)$ is uniformly convergent
    in $I_m$, and therefore also in $X$.
\end{proof}

For a positive element $e$ in a lattice-ordered algebra $X$,
denote by $IA_e$ the smallest subspace of $X$ containing $e$ that is,
at the same time, an order ideal and a subalgebra.

\begin{lem}\label{lem:prelim3}
    Let $X$ be an Archimedean lattice-ordered algebra and let $e \in X_+$ be such
    that $e^{m}\in I_{e^{m+1}}$ for all $m \in \N$. Then
    \[
    IA_e=\bigcup_{m \in \N} I_{e^{m}}.
    \]
\end{lem}
\begin{proof}
    Let $B$ be the set on the right hand side. The condition
    $e^{m}\in I_{e^{m+1}}$ implies that $I_{e^{m}}\subseteq
    I_{e^{m+1}}$ for all $m \in \N$. It follows that $B$, being a
    nested union of ideals, is an ideal.
    To show that it is an algebra, let $x$ and $y$ be elements of $B$, say $|x|\le
    M_1e^{m_1}$ and $|y|\le M_2e^{m_2}$, for some $M_1,M_2 \in \R_+$
    and $m_1,m_2 \in \N$. Then
    \[
    |xy|\le |x| |y|\le M_1M_2 e^{m_1+m_2}
    \]
    and therefore $xy \in I_{e^{m_1+m_2}}\subseteq B$. Hence $B$ is a
    subalgebra and an order ideal that contains $e$. In particular,
    $IA_e\subseteq B$. Conversely, since $IA_e$ is a subalgebra
    containing $e$, $e^{m}\in IA_e$ and, since it is an order ideal,
    $I_{e^{m}}\subseteq IA_e$ for every $m \in \N$. It follows that
    $B\subseteq IA_e$.
\end{proof}

\begin{rem}
    If $X$ has identity $1_X$, then the previous lemma applies to every
    $e\ge 1_X$.
\end{rem}

The following estimate is used for the uniqueness clause.

\begin{lem}\label{lem:prelim4}
    Let $f \in PG_n$, for some $n \in \N$. Then there exists $m \in
    \N$ such that, for every $\varepsilon >0$, one can find $g \in
    \VLA\{\one,\pi _1,\ldots ,\pi _n\}$ with
    \[
    |f-g|\le \varepsilon d^{m}.
    \]
\end{lem}
\begin{proof}
    The idea is to approximate $f$ using
    Stone--Weierstrass on the compact set where $d$ is small, and to
    approximate it by zero where $d$ is big (actually, since in this
    case the error margin $\varepsilon d^{m}$ is big, this can hardly be
    called an approximation). Since $f$ has polynomial growth, there exist $C>0$ and
    $N \in \N$ such that $|f|\le C d^{N}$.  Given $\varepsilon >0$,
    let $R\ge 1$ be a large number, to be determined later. By the
    Stone--Weierstrass theorem, there is a polynomial function $p$
    such that $|f-p|\le \varepsilon/2$ on the compact set $\{d\le
    R+1\}$. Define
    \[
    \chi _R=\one \wedge ( (R+1)\one - d)_+.
    \]
    Notice that $0\le \chi _R\le \one$, $\chi _R$ equals one
    on $\{d\le R\}$, and equals zero outside $\{d\le R+1\}$. Set $g=\chi _Rp$. This function certainly belongs
    to $\VLA\{\one, \pi _1,\ldots ,\pi _n\}$. Moreover,
    \[
    |f-g|\le |\chi _R(f-p)|+|(\one-\chi _R)f|\le \varepsilon
    /2+|\one-\chi _R|Cd^{N}\le \bigg(\varepsilon /2+\frac{C}{R}\bigg)d^{N+1},
    \]
    where in the second inequality we are using that $\chi _R$ is
    supported on $\{d\le R+1\}$, while in the third we are using that $\one-\chi _R$ is supported on
    $\{d>R\}$. By choosing $R$ so that $C/R<\varepsilon /2$, the
    estimate $|f-g|\le \varepsilon d^{N+1}$ holds.
\end{proof}

\begin{proof}[Proof of \cref{thm:conditions}]
    First we show that \textit{(i)} implies \textit{(ii)}. Let
    $\Psx\colon PG_n\to X$ be a lattice-algebra homomorphism with
    $\Psx(\one)=1_X$ and $\Psx(\pi _i)=x_i$ for $i=1,\ldots ,n$. Put
    $Y=\Psx(PG_n)$. Since $PG_n$ is an \falg, and $\Psx$ is a
    lattice-algebra homomorphism, it follows that $Y$ is an
    $f\!$-subalgebra of $X$ (see \cite[Proposition
    3.2]{boulabiar2002}). Denote, as usual, $d=\one \vee |\pi
    _1|\vee \cdots \vee |\pi _n|$ and $e=1_X\vee |x_1|\vee \cdots \vee
    |x_n|$. It is clear that
    $\Psx(d^{m})=e^{m}$. Also, since $PG_n$ is uniformly complete,
    $(I_{d^{m}},\|{\cdot }\|_{d^{m}})$ is complete. By
    \cref{lem:prelim1}, $\|{\cdot }\|_{e^{m}}$ is complete on $Y\cap
    I_{e^{m}}$.

    That \textit{(ii)} implies \textit{(iii)} is obvious. To show that
    \textit{(iii)} implies \textit{(i)}, let $f \in X_+$ be such that
    $f\ge 1_X\vee |x_1|\vee \cdots \vee |x_n|$, and let $Y$ be
    an $f\!$-subalgebra of $X$ satisfying the conditions of
    \textit{(iii)}. Put $Y'=IA_f\cap Y$. Note that $1_X,x_1,\ldots
    ,x_n \in Y'$. Being the intersection of
    sublattice-algebras, $Y'$ is again a sublattice-algebra of $X$. In
    fact, since $Y'\subseteq Y$, $Y'$ is also an $f\!$-subalgebra of
    $X$. By assumption, $f\ge 1_X$, and by \cref{lem:prelim3},
    \[
    IA_f=\bigcup_{m \in \N} I_{f^{m}}.
    \]
    Therefore
    \[
    Y'=Y'\cap IA_f=\bigcup_{m \in \N} Y'\cap I_{f^{m}}.
    \]
    Note that, for every $m \in \N$, $Y'\cap I_{f^{m}}$ is an ideal in
    $Y'$ and $Y'\cap I_{f^{m}}\subseteq Y'\cap I_{f^{m+1}}$. By
    assumption, $\|{\cdot }\|_{f^{m}}$ is complete on $Y\cap
    I_{f^{m}}=Y'\cap I_{f^{m}}$; being a Banach
    lattice, $Y'\cap I_{f^{m}}$ is uniformly complete. By \cref{lem:prelim2}, $Y'$ is
    uniformly complete. In conclusion, $Y'$ is a uniformly complete
    \falg\ with identity $1_X$ that contains $x_1,\ldots ,x_n$. By
    \cref{prop:ucPG}, there exists a lattice-algebra homomorphism
    $\Psx\colon PG_n\to Y'$ satisfying $\Psx(\one)=1_X$ and $\Psx(\pi
    _i)=x_i$ for $i=1,\ldots ,n$. The composition of this map with the
    natural inclusion $Y'\subseteq X$ proves that \textit{(i)} holds.

    It only remains to prove uniqueness.
    Suppose that we are given two maps $\Psx$ and $\Psx'$ satisfying
    the conditions in \textit{(i)}. It is immediate that $\Psx$ and
    $\Psx'$ must coincide on $\VLA\{\one,\pi _1,\ldots ,\pi _n\}$. Fix
    $f \in PG_n$. According to
    \cref{lem:prelim4}, there exists $m \in \N$ such that, for any
    given $\varepsilon >0$, we can find $g \in
    \VLA\{\one,\pi _1,\ldots ,\pi _n\}$ satisfying $|f-g|\le
    \varepsilon d^{m}$. Since $\Psx$ is a lattice homomorphism, and
    $\Psx(d^{m})=e^{m}$, it immediately follows that
    $|\Psx(f)-\Psx(g)|\le \varepsilon e^{m}$. The same estimate holds for
    $\Psx'$. Hence, using that $\Psx(g)=\Psx'(g)$,
    \[
    |\Psx(f)-\Psx'(f)|\le |\Psx(f)-\Psx(g)|+|\Psx'(g)-\Psx'(f)|\le
    2\varepsilon e^{m}.
    \]
    Since $\varepsilon >0$ was chosen arbitrarily, and $m$ only
    depends on the fixed $f$ and not on $\varepsilon >0$, the Archimedean
    property implies $\Psx(f)=\Psx'(f)$. This holds for every $f \in
    PG_n$, thus proving that $\Psx=\Psx'$.
\end{proof}

As a consequence of this theorem, we obtain a characterization of
the lattice-ordered algebras with identity that admit a polynomial growth
continuous function calculus.

\begin{cor}\label{cor:conditions}
    Let $X$ be an Archimedean lattice-ordered algebra with identity
    $1_X$. Then
    $X$ admits a polynomial growth continuous function calculus if and
    only if $X$ is an \falg\ and, for every $x_1,\ldots ,x_n \in X$,
    there exists $f\ge 1_X\vee |x_1|\vee \cdots
    \vee |x_n|$ such that, for every $m \in \N$, the norm
    $\|{\cdot }\|_{f^{m}}$ is complete on the closed sublattice
    generated by $x_1,\ldots ,x_n$ in $I_{f^{m}}$.

    In this case, the polynomial growth continuous function calculus
    is unique (in the sense that for each $n \in \N$ and $\bm x \in
    X^{n}$ there is only one lattice-algebra homomorphism $\Psx\colon
    PG_n\to X$ which satisfies \eqref{eq:proj_one}), and
    \begin{equation}\label{eq:compo}
        \Psi _{(\Psi _{\bm x}(f_1),\ldots ,\Psi _{\bm x}(f_m))}(g)=\Psi
        _{\bm x}(g\circ (f_1\times \cdots \times f_m))
    \end{equation}
    for each $m,n \in \N$, $\bm x \in X^{n}$, $f_1,\ldots ,f_m \in
    PG_n$ and $g \in PG_{m}$, where $f_1\times \cdots \times f_m\colon
    \R^{n}\to \R^{m}$ is the function defined by
    \[
        (f_1\times \cdots \times f_m)(t)=(f_1(t),\ldots ,f_m(t)),\quad
        t \in \R^{n}.
    \]
\end{cor}
\begin{proof}
    If $X$ admits a polynomial growth continuous function calculus,
    then $X$ is an \falg\ by \cref{prop:admitPG}. And if $x_1,\ldots ,x_n \in X$, then
    $e=1_X\vee|x_1|\vee \cdots \vee |x_n| $ has the stated properties
    by \cref{thm:conditions}. The converse also follows directly from \cref{thm:conditions}.

    Uniqueness was proved in \cref{thm:conditions}. To show
    \eqref{eq:compo}, let $f_1,\ldots ,f_m \in PG_n$ and $g \in PG_m$
    be as stated. First we need to check that $g\circ (f_1\times
    \cdots \times f_m) \in PG_n$. Suppose $|g|\le M (\one + |\pi _1| + \cdots +
    |\pi _m|)^{N}$ for some $M\in \R_+$ and $N \in \N$. Then it is immediate
    to check that
    \[
    |g\circ (f_1\times \cdots \times f_m)|\le M(\one + |f_1| + \cdots
    + |f_m|)^{N}.
    \]
    Since $f_1,\ldots ,f_m \in PG_n$, the right hand side of this
    equation belongs to $PG_n$. Since $PG_n$ is an ideal in
    $C(\R^{n})$, and $g\circ (f_1\times \cdots \times f_m)$ is
    continuous, it follows that $g\circ (f_1\times \cdots \times f_m)
    \in PG_n$. Hence, it makes sense to evaluate $\Psx$ at $g\circ (f_1\times
    \cdots \times f_m)$. This is true of every $g \in PG_m$.
    It is immediate to check that the map
    \[
    \begin{array}{cccc}
            & PG_m & \longrightarrow & X \\
            & g & \longmapsto & \Psx(g\circ (f_1\times \cdots \times
            f_m)) \\
    \end{array}
    \]
    is a lattice-algebra homomorphism that maps $\pi _i$ to $\Psx
    (f_i)$ for $i=1,\ldots ,m$. By uniqueness, it must be the map $\Psi
    _{(\Psx(f_1),\ldots ,\Psx(f_m))}$.
\end{proof}

To complete this section, we present three examples illustrating \cref{thm:conditions}.
The first one shows that, even if $X$ is a uniformly complete
lattice-ordered algebra with identity, it may not contain sufficiently many
$f\!$-subalgebras to admit a polynomial growth continuous function
calculus.

\begin{example}
    Consider $\R^2$ with the usual vector lattice structure. Define
    the product
    \[
        (x,y)(x',y')=(xx'+xy'+yx',yy').
    \]
    It is immediate to check that $\R^2$ with this product is a
    lattice-ordered algebra with identity $(0,1)$. But it is not an \falg,
    because $(1,0)\wedge (0,1)=0$ while $(1,0)(0,1)=(1,0)\neq 0$. The
    element $(1,0)$ is such that, if
    $Y$ is a sublattice-algebra containing both $(1,0)$ and the
    identity $(0,1)$, then
    $Y=\R^2$, which is not an \falg. Hence there is no $f$-subalgebra
    containing both $(1,0)$ and the identity. By
    \cref{thm:conditions}, we cannot perform
    polynomial growth continuous function calculus on $(1,0)$.
\end{example}

The second example fails the other condition: it is an \falg, but
it does not contain sublattice-algebras $Y$ for which the norm $\|{\cdot
}\|_{e^{m}}$ is complete on $Y\cap I_{e^{m}}$ (where we are using the notation of \cref{thm:conditions}).

\begin{example}
    Consider the function $I\colon [0,1]\to [0,1]$ given by $I(x)=x$
    for all $x \in [0,1]$. Let $X$ be the sublattice-algebra of
    $C[0,1]$ generated by $I$ and $\one_{[0,1]}$. Note that $X$ is an
    \falg\ with identity. We have $\one_{[0,1]}\vee I=\one_{[0,1]}$
    and the ideal generated by $\one_{[0,1]}$ in $X$ is $X$.
    Yet $X$ is not complete for $\|{\cdot }\|_{\one_{[0,1]}}$. Indeed,
    this norm is precisely the uniform norm, and $X$ is not closed
    in the uniform norm by the Stone--Weierstrass theorem.
\end{example}

The final example is a positive one: it shows how
\cref{cor:conditions} can be used in practice to prove that a
non-uniformly complete lattice-ordered algebra with identity admits a
polynomial growth continuous function calculus. This example is a
continuation of \cite[Proposition 3.1]{laustsen_troitsky2020}.

\begin{example}
    Let $K$ be a compact Hausdorff space of infinite cardinality, and
    let $s_0\in K$ be an accumulation point of a countably infinite
    subset of $K$. (For instance, $K=[0,1]$ and $s_0=0$.) Let
    \[
    X=\{\, f\in C(K) : f\text{ is constant on some neighborhood of
    }s_0 \, \}.
    \]
    It was shown in \cite[Proposition 3.1]{laustsen_troitsky2020} that
    this set is a proper and dense sublattice of $C(K)$ that is not
    uniformly complete. It is easy to check that it is a
    subalgebra and that $\one_K \in X$. Therefore, $X$ is an Archimedean
    \falg\ with identity. We will show that it admits a polynomial
    growth continuous function calculus by checking the condition of
    \cref{cor:conditions}.

    Let $f_1,\ldots ,f_n \in X$ and let $U$ be a neighborhood of
    $s_0$ in which $f_1|_U,\ldots ,f_n|_{U}$ are all constant. Let
    $c=\max\{\|f_1\|_\infty ,\ldots ,\|f_n\|_\infty ,1\}$, and let
    \[
    Y=\{\, f \in C(K) : f|_U\text{ is constant} \, \}.
    \]
    Then $e=c\one_K$ satisfies $e\ge \one_K \vee |f_1|\vee \cdots \vee
    |f_n|$ and $f_1,\ldots ,f_n \in Y$. Note that $Y$ is a closed
    sublattice-algebra of $C(K)$. For fixed $m \in \N$,
    $I_{e^{m}}=X$ and $\|{\cdot }\|_{e^{m}}$ is a positive multiple of
    the uniform norm;
    it follows directly that
    $\|{\cdot }\|_{e^{m}}$ is complete on $I_{e^{m}}\cap Y$. Hence,
    even though $X$ is not uniformly complete, it admits a polynomial
    growth continuous function calculus.
\end{example}

\section{Vector spaces with polynomial growth continuous function
calculus}\label{sec:final}

Our final result is an extension of \cite[Theorem
1.4]{laustsen_troitsky2020}. It states that if a
vector space admits a sufficiently strong polynomial growth
continuous function calculus without constants, then it is automatically a commutative \falg.

\begin{thm}\label{thm:vspace}
    Let $X$ be a vector space and suppose that, for each $n \in \N$
    and each $\bm x = (x_1,\ldots ,x_n)\in X^{n}$, there is a linear
    map $\Psx\colon PG_n^{h}\to X$ which satisfies \eqref{eq:proj} and
    \eqref{eq:compo} (with $PG_n^{h}$ and $PG_m^{h}$ in place of
    $PG_n$ and $PG_m$). Then $X$
    can be endowed with the structure of a commutative \falg\ that
    admits a polynomial growth continuous function
    calculus without constants.
\end{thm}
\begin{proof}
    For $n \in \N$ and $f \in PG_n^{h}$, denote $\Psx(f)$ by $f(\bm x)$. In this
    notation, the linearity of $\Psx$ reads
    \[
        (f+\lambda g)(\bm x)=f(\bm x)+\lambda g(\bm x)
    \]
    for all $f,g \in PG_n^{h}$ and $\lambda \in \R$. Condition
    \eqref{eq:proj}
    becomes $\pi _i(\bm x)=x_i$, while \eqref{eq:compo} has the form
    \[
    g(f_1(\bm x),\ldots ,f_m(\bm x))=(g\circ (f_1\times \cdots \times
    f_m))(\bm x).
    \]

    Consider the function of $PG_2^{h}$:
    \[
    \begin{array}{cccc}
    \sigma \colon& \R^2 & \longrightarrow & \R \\
            & (t_1,t_2) & \longmapsto & t_1\vee t_2 \\
    \end{array}
    \]
    and define a partial order $\le $ on $X$ by setting $x_1\le x_2$
    if and only if $\sigma (x_1,x_2)=x_2$. It is shown in
    \cite[Theorem 1.4]{laustsen_troitsky2020} that this order makes
    $X$ into a vector lattice. Moreover, the supremum of $x_1,x_2 \in
    X$ in this order is given by $\sigma (x_1,x_2)$.

    Consider the function of $PG_2^{h}$:
    \[
    \begin{array}{cccc}
    \rho \colon& \R^2 & \longrightarrow & \R \\
            & (t_1,t_2) & \longmapsto & t_1t_2 \\
    \end{array}.
    \]
    Define a product on $X$ by
    \[
    \begin{array}{cccc}
            & X\times X & \longrightarrow & X \\
            & (x_1,x_2) & \longmapsto & \rho (x_1,x_2) \\
    \end{array}.
    \]
    We claim that $X$, with this product, is a commutative real
    algebra. Let us check the required axioms.

    \textsc{Commutativity.} Since
    the product in $\R$ is commutative, $\rho (t_1,t_2)=\rho
    (t_2,t_1)$ holds for all $t_1,t_2 \in \R$. In other words, the identity $\rho
    \circ \pi _1\times \pi _2=\rho \circ \pi _2\times \pi _1$ holds in
    $PG_2^{h}$. Using \eqref{eq:proj} and \eqref{eq:compo}, the following holds for all
    $x_1,x_2\in X$:
    \[
        (\rho \circ \pi _1\times \pi _2)(x_1,x_2)=\rho (\pi
        _1(x_1,x_2),\pi _2(x_1,x_2))=\rho (x_1,x_2).
    \]
    Similarly, $(\rho \circ \pi _2\times \pi _1)(x_1,x_2)=\rho
    (x_2,x_1)$. Hence $\rho(x_1,x_2)=\rho (x_2,x_1)$.

    \textsc{Associativity.} Associativity in $\R$
    means that $\rho (\rho (t_1,t_2),t_3)=\rho (t_1,\rho (t_2,t_3))$
    for all $t_1,t_2,t_3\in \R$. This, in turn, implies that the
    following identity holds in $PG_3^{h}$:
    \[
        \rho \circ [(\rho \circ \pi _1\times \pi _2)\times \pi _3]=
        \rho \circ [\pi _1 \times (\rho \circ \pi _2\times \pi _3)].
    \]
    Applying this to $(x_1,x_2,x_3) \in X^{3}$ and simplifying
    with \eqref{eq:proj} and \eqref{eq:compo} yields
    \[
        \rho (\rho (x_1,x_2),x_3)=\rho (x_1,\rho (x_2,x_3)).
    \]

    \textsc{Distributivity.} Distributivity in $\R$ translates into
    the identity
    \[
        \rho \circ [\pi _1\times (\pi _2+\pi _3)]=\rho \circ \pi
        _1\times \pi _2+\rho \circ \pi _1\times \pi _3.
    \]
    Applying this to $(x_1,x_2,x_3) \in X^{3}$, using linearity and
applying \eqref{eq:proj} and \eqref{eq:compo} yields $\rho (x_1,x_2+x_3)=\rho (x_1,x_2)+\rho (x_1,x_3)$.
    Distributivity on the right follows from commutativity and
    distributivity on the left.

    \textsc{Associativity with scalars.} For fixed $\lambda ,\mu \in \R$,
    associativity and commutativity in $\R$ imply $\rho (\lambda t_1,\mu t_2)=\lambda
    \mu \rho (t_1,t_2)$ for all $t_1,t_2 \in \R$. This means that
    \[
    \rho \circ (\lambda \pi _1) \times (\mu \pi _2)=(\lambda \mu) (\rho
    \circ \pi _1\times \pi _2).
    \]
    Evaluating at $(x_1,x_2)\in X^2$ and using linearity and \eqref{eq:proj} and \eqref{eq:compo}
    gives $\rho (\lambda x_1,\mu x_2)=\lambda \mu \rho (x_1,x_2)$.

    Hence $X$ is a real algebra when equipped with the product
    $x_1x_2=\rho (x_1,x_2)$. Next we show that it is also a
    lattice-ordered algebra. For this we have to check that, if $x_1,x_2 \in
    X_+$, then $x_1x_2\in X_+$. Since in $\R$ the product of positive
    elements is positive, the identity
    \[
        [(t_1\vee 0)(t_2\vee 0)]\vee 0=(t_1\vee 0)(t_2\vee 0)
    \]
    holds for all $t_1,t_2 \in \R$. This implies that
    \[
    \sigma \circ [\rho \circ (\sigma \circ \pi _1\times \zero)\times
    (\sigma \circ \pi _2\times \zero)]=\rho \circ [(\sigma \circ \pi
    _1\times \zero)\times (\sigma \circ \pi _2\times \zero)],
    \]
    where $\zero \colon \R^{2}\to \R$ denotes the constant zero
    function. Linearity of $\Psx$ implies that $\zero(x_1,x_2)=0$ for
    all $(x_1,x_2)\in X^2$. Therefore the evaluation of the previous
    identity at $(x_1,x_2)$ yields
    \[
        \sigma [\rho (\sigma (x_1,0),\sigma (x_2,0)),0]=\rho [\sigma
        (x_1,0),\sigma (x_2,0)].
    \]
    This equation holds for all $x_1,x_2 \in X$. When $x_1,x_2\ge 0$
    (that is, when $\sigma (x_i,0)=x_i$ for $i=1,2$) this equation
    becomes
    \[
        \sigma [\rho (x_1,x_2),0]=\rho (x_1,x_2).
    \]
    This is, by definition, the same as $x_1x_2=\rho (x_1,x_2)\ge 0$.
    Hence $X$ is a lattice-ordered algebra.

    To show that $X$ is an \falg, we will use the following fact due
    to G.\ Birkhoff \cite[Lemma 3]{birkhoff1967}: the lattice-ordered algebra $X$ is an
    \falg\ if and only if the equalities
    \[
        (x_1\vee 0)\wedge [((-x_1)\vee 0)(x_2\vee 0)]=0=
        (x_1\vee 0)\wedge [(x_2\vee 0)((-x_1)\vee 0)]
    \]
    hold for all $(x_1,x_2)\in X^2$. Since $X$ is commutative, we
    need only check the identity on the left. Since $\R$ is an \falg,
    \[
        (t_1\vee 0)\wedge [((-t_1)\vee 0)(t_2\vee 0)]=0
    \]
    holds for all $t_1,t_2 \in \R$; let us rewrite it as
    \[
        (-(t_1\vee 0))\vee  [-((-t_1)\vee 0)(t_2\vee 0)]=0
    \]
    so that only suprema appear in the formula. This implies that
    \[
        \sigma \circ [[-\sigma \circ (\pi _1\times \zero)]\times [-\rho
        \circ ((\sigma \circ ((-\pi _1)\times \zero))\times (\sigma \circ
        (\pi _2\times \zero)))]]=\zero.
    \]
    Applying this identity to $(x_1,x_2)\in X^2$ yields
    \[
    \sigma (-\sigma (x_1,0),-\rho (\sigma (-x_1,0),\sigma (x_2,0)))=0.
    \]
    By \cite[Lemma 3]{birkhoff1967}, $X$ is a commutative \falg.

    Finally, to show that $X$ admits a polynomial growth
    continuous function calculus without constants, it suffices to check that
    $\Psx\colon PG_n^{h}\to X$ is a lattice-algebra homomorphism for
    all $\bm x=(x_1,\ldots ,x_n) \in X^{n}$. To make the equations
    clearer, denote $\sigma (x_1,x_2)$ by $x_1\vee x_2$ and $\rho
    (x_1,x_2)$ by $x_1x_2$. Then for all $f, g \in PG_n^{h}$:
    \[
    \Psx(f\vee g)=\Psx(\sigma \circ (f\times g))=(\sigma \circ (f\times g))(\bm x)=\sigma (f(\bm
    x),g(\bm x))=\Psx(f)\vee \Psx(g),
    \]
    and also
    \[
    \Psx(fg)=\Psx(\rho \circ (f\times g))=(\rho \circ (f\times g))(\bm
    x)=\rho (f(\bm x),g(\bm x))=\Psx(f)\Psx(g).\qedhere
    \]
\end{proof}

The \falg\ produced by \cref{thm:vspace} need not be Archimedean, as
the following example shows.

\begin{example}
    Let $\R^{\N}$ be the \falg\ of all real sequences, endowed with
    pointwise operations and order, let $\one_{\N}=(1,1,1,\ldots )$
    denote its identity, and let $c_{00}$ be the set of finitely
    supported sequences. Since $c_{00}$ is at the same time an order
    and a ring ideal of $\R^{\N}$, the quotient
    \[
        X=\R^{\N}/c_{00}
    \]
    is an \falg\ with identity, and the
    quotient map $q\colon \R^{\N}\to X$ is a surjective lattice-algebra
    homomorphism. For $u \in \R^{\N}$, let us denote by $u(k)$ the
    $k$-th coordinate of $u$. Notice that $q(u)\le q(v)$ if and only
    if $u(k)\le v(k)$ for all but finitely many $k \in \N$.
    With this in mind, it is not difficult to check that $X$ is not
    Archimedean: the element $v \in \R^{\N}$ given by $v(k)=k$ ($k \in
    \N$) satisfies $n q(\one_\N)\le q(v)$ for all $n \in \N$, and yet
    $q(\one_\N)>0$.

    We claim that the underlying vector space of $X$ satisfies the
    hypotheses of \cref{thm:vspace}. The \falg\ $\R^{\N}$ is
    Archimedean and uniformly complete. By
    \cref{prop:ucPG}, $\R^{\N}$ admits a polynomial growth continuous
    function calculus, and by \cref{cor:conditions} this calculus
    satisfies \eqref{eq:compo}. Given $\bm u=(u_{1},\ldots ,u_{n})\in
    (\R^{\N})^{n}$, we denote by $\Psi _{\bm u}\colon PG_n\to \R^{\N}$
    the associated lattice-algebra homomorphism.

    Actually, the calculus of $\R^{\N}$ is given by
    \begin{equation}\label{eq:coordinatewise}
        \Psi _{\bm u}(f)=\big( f(u_{1}(k),\ldots ,u_{n}(k)) \big)_{k \in
        \N},\qquad \text{where }f \in PG_n.
    \end{equation}
    Indeed, the right hand side of \eqref{eq:coordinatewise} defines a
    lattice-algebra homomorphism $PG_n\to \R^{\N}$, because the
    operations of $PG_n$ and of $\R^{\N}$ are computed pointwise.
    Furthermore, it
    sends $\one$ to $\one_{\N}$ and $\pi _i$ to $u_{i}$, hence it
    coincides with $\Psi _{\bm u}$ by uniqueness.

    Let now $\bm x=(x_1,\ldots ,x_n)\in X^{n}$. If $\bm u$ and $\tilde{\bm u}$ are two lifts of
    $\bm x$, meaning that $q(u_i)=q(\tilde u_i)=x_i$ for $i=1,\ldots
    ,n$, then $u_{i}$ and $\tilde u_{i}$ differ in finitely many
    coordinates for each $i$, so \eqref{eq:coordinatewise} shows that
    $\Psi _{\bm u}(f)$ and $\Psi _{\tilde{\bm u}}(f)$ differ in
    finitely many coordinates as well, for every $f \in PG_n$. Hence
    $q\circ \Psi _{\bm u}=q\circ \Psi _{\tilde{\bm u}}$. In other
    words, we may define a lattice-algebra homomorphism $\Psx\colon
    PG_n\to X$ satisfying \eqref{eq:proj} by setting $\Psx=q\circ\Psi
    _{\bm u}$, where $\bm u$ is \emph{any} lift of $\bm x$.

    It remains to check that the $\Psx$ thus defined also satisfies \eqref{eq:compo}. Let $f_1,\ldots ,f_m \in
    PG_n$, let $g \in PG_m$, and let $\bm u$ be a lift of $\bm x$. By
    the definition of $\Psx$, the tuple $(\Psi _{\bm u}(f_1),\ldots
    ,\Psi _{\bm u}(f_m))$ is a lift of $(\Psx(f_1),\ldots ,\Psx(f_m))$,
    and therefore
    \begin{align*}
        \Psi _{(\Psx(f_1),\ldots ,\Psx(f_m))}(g)
        &=q\big( \Psi _{(\Psi _{\bm u}(f_1),\ldots ,\Psi _{\bm
        u}(f_m))}(g) \big)\\
        &=q\big( \Psi _{\bm u}(g\circ (f_1\times \cdots \times f_m))
        \big)\\
        &=\Psx\big( g\circ (f_1\times \cdots \times f_m) \big),
    \end{align*}
    where the middle equality is precisely \eqref{eq:compo} for the
    calculus of $\R^{\N}$.

    In particular, the restrictions $\Psx|_{PG_n^{h}}$ satisfy the
    hypotheses of \cref{thm:vspace}. Moreover, the lattice operations
    that \cref{thm:vspace} introduces on $X$ are the ones $X$ already
    carries: since $\sigma =\pi _1\vee \pi _2$ and $\Psi _{\bm u}$ is a
    lattice homomorphism,
    \[
        \Psi _{(x_1,x_2)}(\sigma )=q\Psi _{(u_1,u_2)}(\sigma )=q(u_{1}\vee u_{2})=x_1\vee x_2.
    \]
    Hence the \falg\ structure that \cref{thm:vspace} constructs on
    $X$ is non-Archimedean.
\end{example}

\section*{Acknowledgments}

The author is indebted to the anonymous referee for finding a gap in
the original proof of the uniqueness clause in \cref{thm:conditions}.

\section*{AI disclosure}

Claude Opus 5 was used to find references and to correct minor errors.
The author is solely responsible for the results and the writing of this paper.

\printbibliography

\end{document}